\documentclass[12pt,a4paper]{amsart}
\usepackage{mathrsfs}
\usepackage{amssymb}
\usepackage{amsxtra}
\usepackage{cleveref,comment,color}
\usepackage[T1]{fontenc}
\usepackage[utf8]{inputenc}

\newcommand{\pG}{\mathcal{G}}
\newcommand{\pS}{\mathcal{S}}
\newcommand{\eP}{\mathscr{P}}
\newcommand{\eT}{\mathscr{T}}
\newcommand{\eM}{\mathscr{M}}
\newcommand{\eN}{\mathscr{N}}
\newcommand{\bN}{\mathbb{N}}
\DeclareMathOperator{\diam}{diam}
\newtheorem{theorem}{Theorem}
\newtheorem{proposition}[theorem]{Proposition}
\newtheorem{lemma}[theorem]{Lemma}
\newtheorem{corollary}[theorem]{Corollary}
\theoremstyle{definition}

\numberwithin{equation}{section}
\begin{document}
\title[On the diameter of a super-order-commuting graph]{On the diameter of a super-order-commuting graph}

\author[J. Bra\v{c}i\v{c}]{Janko Bra\v{c}i\v{c}}
\address{Faculty of Natural Sciences and Engineering, University of Ljubljana, A\v{s}ker\v{c}eva c. 12, SI-1000 Ljubljana, Slovenia}
\email{janko.bracic@ntf.uni-lj.si}
\author[B. Kuzma]{Bojan Kuzma}
\address{${}^1$University of Primorska, Glagolja\v{s}ka 8, SI-6000 Koper, Slovenia and ${}^2$IMFM, Jadranska 19,  SI-1000 Ljubljana, Slovenia}
\email{bojan.kuzma@famnit.upr.si}


\begin{abstract}
We answer a question about the diameter of an order-super-commuting graph on a symmetric group by studying the number-theoretical concept
of  $d$-complete sequences of primes in arithmetic progression.
\end{abstract}
\maketitle

\section{Introduction} \label{Sec01}
\setcounter{theorem}{0}

In a recent paper \cite{Dalal-Mukherjee-Patra} the authors studied the properties of super graphs on finite groups. These graphs are based
on already well-studied commuting, power,  and enhanced power graphs (we refer to \cite{Cameron} for a nice survey and the current
development), but with an equivalence relation thrown in. More precisely, if $\sim$ is an equivalent relation on a graph  $\Gamma$, then
$\sim$-super-$\Gamma$ has the same vertex set as $\Gamma$ however, its edge set is enlarged, whereby distinct $x,y$ form an edge if there
exist $u\sim x$ and $v\sim y$ with either $u=v$ or $(u,v)\in E(\Gamma)$.
In particular, with the order relation on a finite group $\pG$ (i.e., $x\sim y$ if $x$ and $y$  have the same order), the
order-super-commuting graph, $\Delta^o(\pG)$, of a group $\pG$
is a simple graph
$\Delta^o(\pG)$
with the vertex set equal to
$\pG$ and
where two disjoint vertices $x,y$ form an edge if there exists commuting $u,v$ with $|u|=|x|$ and $|v|=|y|$ (here, we also allow $u=v$, so
in particular, each conjugacy class forms an induced complete graph). We should caution that in \cite{Dalal-Mukherjee-Patra} the central
elements and in particular, the identity also belong to the vertex set of the commuting graph, but we do not allow the loops. Notice that
this contrasts with a similar definition in some of the existing literature~\cite{Akbari,Giudici-Parker}, where the central elements are
removed. It was shown in \cite{Dalal-Mukherjee-Patra} that, for $n\ge 4$, the identity, that is, the center of the group, is the only
dominant vertex of $\Delta^o(\pS_n)$, the order-super-commuting graph of the symmetric group $\pS_n$ on $n$ elements. By deleting all the
dominant vertices, one obtains the reduced graph, $\Delta^o(\pS_n)^\ast$. This is connected if and only if neither $n-1$ nor $n$ is a prime
number; moreover, if it is disconnected, then it has exactly two components, and if it is connected, then its diameter is bounded above by
$3$, see \cite[Proposition 4.9 and Theorem 4.11]{Dalal-Mukherjee-Patra}. Whether its diameter is $3$ or smaller was
not determined for all values of $n\ge 4$, but it was shown \cite[Proposition 4.12]{Dalal-Mukherjee-Patra} that
the following holds.

\begin{proposition}\label{prop}
Let $n\ge 4$. If neither $n$ nor $n-1$ is a prime number, then $\diam\Delta^o(\pS_n)^\ast=3$ if and only if there exist
nonempty disjoint subsets  $\eT_1,\eT_2$ consisting of primes smaller or equal to $n$, such that, for some positive
integers $\alpha_p$ and $\beta_q$ we have
$$M^\alpha_{\eT_1}:=\sum_{p\in \eT_1} p^{\alpha_p}\in\{n-1,n\},\qquad
M^\beta_{\eT_2}:=\sum_{q\in \eT_2} q^{\beta_q}\le n,$$
and $p+M^\beta_{\eT_2}>n$, for every $p\in \eT_1$.\qed
\end{proposition}

With this proposition, the authors managed to prove that the diameter is $3$
if $n$ or $n-1$ is a nontrivial power of a prime or a sum of two prime
powers, with both primes distinct and greater or equal to $5$. The latter
provided that the strong Goldbach conjecture is true, immediately yields that
$\diam\Delta^o(\pS_n)^\ast=3$, for every even integer $n\ge 4$. This paper
aims to give a complete solution, without relying on the strong Goldbach
conjecture and proves that $\diam\Delta^o(\pS_n)^\ast=3$ for every $n\ge 4$
such that neither $n$ nor $n-1$ is a prime number. Our main ingredient is the
fact that the sequences consisting of primes congruent to $\pm 1$ modulo $4$
are complete (see \Cref{theo01} and its consequences).

An infinite sequence of distinct positive integers $\{ a_n;\; n\in \bN\}$ is
called complete if every sufficiently large positive integer is a sum of
distinct $a_i$ (sometimes such sequences are called weakly
complete, while the term complete sequence is reserved for the case when
every integer is a sum of distinct $a_i$). Erd\H{o}s and Lewin \cite{EL}
call a complete sequence $d$-complete if every sufficiently large integer is
a sum of distinct $a_i$ such that no one divides the other. In \cite{EL},
there are given several examples of $d$-complete sequences. For instance, it
is proved that, for positive integers $p$ and $q$, the sequence $\{
p^{a}q^{b};\; a, b\in \bN \}$ is $d$-complete if and only if $\{ p,q\}=\{
2,3\}$. Bruckman \cite{Bru} has proved that the sequence $\eP=\{
2,3,5,\ldots\}$ of all prime numbers is $d$-complete. We adjust Bruckman's
proof and show that the sequences of all prime numbers congruent to $1$,
respectively to $3$, modulo $4$ are $d$-complete.

\section{Results} \label{Sec02}
\setcounter{theorem}{0}

\subsection{Generating polynomials}
Let $1\le q_1< q_2< \dots$ be a sequence of integers. For every $n\in \bN$, let
\begin{equation*}
f_{n}(x)=\bigl(1+x^{q_1}\bigr)\cdots \bigl(1+x^{q_n}\bigr).
\end{equation*}
It is clear that $f_n(x)$ is a polynomial of degree $S_n=q_1+\dots+q_n$.
Denote the coefficient of $f_n(x)$ at power $x^m$ by $\gamma_m(n)$; we
also let $\gamma_{m}(n)=0$ if $m\ge S_n+1$. Then
\begin{equation*}
f_n(x)=\sum_{m=0}^{S_n} \gamma_m(n) x^m.
\end{equation*}
Since
\begin{equation*}
\begin{split}
\sum_{m=0}^{S_{n+1}} \gamma_m(n+1) x^m&=f_{n+1}(x)=f_{n}(x)\cdot(1+x^{q_{n+1}})=
\biggl( \sum_{m=0}^{S_n} \gamma_m(n) x^m\biggr) \bigl(1+x^{q_{n+1}}\bigr)\\
&=\sum_{m=0}^{S_n} \gamma_m(n) x^m+
\sum_{m=0}^{S_n} \gamma_m(n) x^{m+q_{n+1}}
\end{split}
\end{equation*}
the comparison of the coefficients gives
\begin{equation} \label{eq01}
\gamma_m(n+1)=\left\{
\begin{array}{cl}
\gamma_m(n),&\quad 0\leq m<q_{n+1};\\
\gamma_m(n)+\gamma_{m-q_{n+1}}(n),&\quad q_{n+1}\leq m\leq S_n;\\
\gamma_{m-q_{n+1}}(n),&\quad S_n< m\leq S_{n+1}.
\end{array} \right.
\end{equation}
It follows from \eqref{eq01} that $\gamma_m(n+1)\geq \gamma_m(n)\geq 0$. On the other hand,
let $m\geq 0$ be arbitrary but fixed. Let $ n\in \bN$ be such that $m<q_{n+1}$. Then, by \eqref{eq01},
$\gamma_m(n+1)=\gamma_m(n)$. Since $m<q_{n+1}<q_{n+2}$ we have
$\gamma_m(n+2)=\gamma_m(n+1)$
and therefore $\gamma_m(n+2)=\gamma_m(n)$. By induction,
$\gamma_m(k)=\gamma_m(n)$, for all $k>n$.
Thus, we may define $\Gamma_m=\max\{ \gamma_m(n);\; n\in \bN\}$.
Note that $\Gamma_m>0$ if and only if there exist $n$ such that $\gamma_m(n)>0$ which is equivalent
to the fact that there exist distinct sequence members  $q_{j_1},\ldots,q_{j_\ell}$
where $\ell\geq 1$, such that $m=q_{j_1}+\cdots+q_{j_\ell}$.

\begin{theorem}  \label{theo02}
Let $1\le q_1<q_2<\cdots$ be a sequence of integers with partial sums $S_n=q_1+\dots+q_n$.
Consider
$$f_n(x)=(1+x^{q_1})\cdots (1+x^{q_n})=\sum_{m=0}^{S_n}\gamma_m(n) x^m.$$
 If there exist positive integers $k_0$ and $n_0$ such that
$$2k_0+q_{n+1}\le S_{n},\quad \text{for}\;\; n\ge n_0,\qquad\text{and}\qquad \gamma_m(n_0)\ge 1,\quad
\text{for}\;\; k_0\le m\le S_{n_0}-k_0,$$
then every integer $n\ge k_0$ is a sum of (one or more) different members of $\{q_1,q_2,\dots\}$.
\end{theorem}

\begin{proof}
We will prove by induction that, for $n\geq n_0$, we have
$$\gamma_m(n)\geq 1\quad\text{if}\;\; k_0\leq m\leq S_{n}-k_0.$$
It is obvious that this will imply $\Gamma_m\geq 1$, for all $m\geq k_0$, and the statement will follow.

Let $\eN\subseteq \bN$ denote the set of all integers $n\geq n_0$ such that $\gamma_m(n)\geq 1$ if
$k_0\leq m\leq S_n-k_0$. By the hypothesis, $n_0\in \eN$.
Assume that $n\in \eN$. If $k_0\leq m\leq S_n-k_0$, then $\gamma_m(n+1)\geq \gamma_m(n)\geq 1$,
by the inductive hypothesis. Similarly, if $m$ is such that $k_0\leq m-q_{n+1}\leq S_n-k_0$, then
$\gamma_{m-q_{n+1}}(n)\geq 1$, again by the inductive hypothesis, and therefore $\gamma_m(n+1)\geq 1$,
by \eqref{eq01}. Clearly, $k_0\leq m-q_{n+1}\leq S_n-k_0$ is equivalent to
$k_0+q_{n+1}\leq m\leq S_n+q_{n+1}-k_0=S_{n+1}-k_0$. Also, by the assumptions, $k_0+q_{n+1}\leq S_n-k_0$
so that the intersection of the intervals  $[k_0,S_n-k_0]$ and $[k_0+q_{n+1},S_{n+1}-k_0]$ is nonempty.
We conclude that $\gamma_m(n+1)\geq 1$, for all $m$ such that $k_0\leq m\leq S_{n+1}-k_0$. Hence, $n+1\in \eN$.
\end{proof}

We next show that there always exists $k_0$ which satisfies the first condition in \Cref{theo02}, provided that
the sequence of integers grows at most exponentially.

\begin{lemma} \label{lem03}
Let $1\le q_1<q_2<\dots$ be a sequence of integers satisfying $q_{n+1}<2 q_n$ and let $k_0\ge 0$ be a given integer.
Then, $S_n-q_{n+1}\ge 2k_0$, for every $n\ge q_1+2k_0+1$.
\end{lemma}

\begin{proof}
Notice first that $q_i<q_{i+1}\le 2q_i-1$, so that  $q_i-q_{i+1}\ge q_i-(2q_i-1)=-q_i+1$.
Then, proceeding backward, we get
\begin{equation*}
 \begin{split}S_i&-q_{i+1}=q_1+q_2+\cdots+(q_i-q_{i+1})\ge q_1+q_2+\cdots+(q_{i-1}-q_i)+1 \geq\\
 &\geq q_1+q_2+\cdots+(q_{i-2}-q_{i-1})+1+1\ge\dots\ge q_1-q_2+(i-2)\ge -q_1+(i-1).
 \end{split}
\end{equation*}
Hence, with $i=q_1+2k_0+1$ we get that $S_i-q_{i+1}\ge 2k_0$. This was the start of the induction.
To finish the proof, let $\eM$ denote the set of all integers $n\geq q_1+2k_0+1$ such that the
statement of the lemma is valid. If $n\in\eM$, then
$$ S_{n+1}-q_{n+2}=S_n-q_{n+1}+2q_{n+1}-q_{n+2}\geq S_n-q_{n+1}\ge 2k_0,$$
and therefore $n+1\in \eM$.
\end{proof}

\subsection{Prime numbers congruent $1$, respectively $3$, modulo $4$}

Let $d$ be a positive integer and let $ 1\leq r<d$ be such that $\gcd(d,r)=1$. The celebrated Dirichlet's Theorem
says that there are infinitely many prime numbers congruent to $r$ modulo $d$.
Let $\eP(d,r)=\{q_{d,r}(1)<q_{d,r}(2)<q_{d,r}(3)<\cdots \}$ be the sequence of all prime numbers congruent to $r$
modulo $d$ and let $S_{d,r}(n)$ denote the sum of the first $n$ prime numbers in $\eP(d,r)$. In what follows, we
are interested in $\eP(4,1)$ and $\eP(4,3)$.

By \cite[Theorem 1]{CH}, for every $x\geq 887$, sets $(x,1.040x]\cap \eP(4,1)$ and $(x,1.040x]\cap \eP(4,3)$
are nonempty, that is, there exist prime numbers congruent to $1$ and to $3$ modulo $4$
in the interval $(x,1.040x]$. Using this and a straightforward verification of the tables of prime numbers congruent
to $1$, respectively to $3$, modulo $4$ gives the following lemma.

\begin{lemma} \label{lem01}
For $x\geq 7$, the interval $(x, 2x]$ contains a prime number congruent to $1$ modulo $4$ as well as
a prime number congruent to $3$ modulo $4$.\qed
\end{lemma}

In the following table, we list the first $15$ prime numbers congruent to $1$, respectively to $3$, modulo $4$,
and their partial sums.

\begin{table}[h!]
{\tiny
\begin{tabular}{|c|| c|c|c|c|c|c|c|c|c|c|c|c|c|c|c|}
\hline
$n$ 		& 1 	& 2 	 & 3   & 4   & 5   & 6  & 7	 & 8   & 9   & 10 & 11 & 12 & 13 & 14 & 15 \\
\hline\hline
$q_{4,1}(n)$ 	& 5 	& 13 & 17 & 29 & 37 & 41	& 53 & 61 & 73 & 89 & 97 & 101 & 109 & 113 & 137 \\
\hline
$S_{4,1}(n)$ 	& 5 	& 18 & 35 & 64 & 101 & 142 & 195 & 256 & 329 & 418 & 515 & 616 & 725 & 838 & 975 \\
\hline \hline
$q_{4,3}(n)$ 	& 3 	& 7	 & 11 & 19 & 23 & 31 & 43 & 47 & 59 & 67 & 71 & 79 & 83 & 103 & 107 \\
\hline
$S_{4,3}(n)$ 	& 3 	& 10 & 21 & 40 & 63 & 94 & 137 & 184 & 243 & 310 & 381 & 460 & 543 & 656 & 753 \\
\hline
\end{tabular}   \vspace{1mm}
\caption{\label{Tab:1}}
}
\end{table}\medskip

\begin{lemma} \label{lem02}
(a) If $n\geq 10$, then $S_{4,1}(n)-q_{4,1}(n+1)\geq 244$.

(b) If $n\geq 8$, then $S_{4,3}(n)-q_{4,3}(n+1)\geq 112$.
\end{lemma}

\begin{proof}
(a) Let $\eM$ denote the set of all integers $n\ge 10$ such that the statement (a)  of the lemma is valid.
The table shows that $S_{4,1}(10)-q_{4,1}(10)=418 -97= 321$, so $10\in\eM$.
To prove the inductive step, let $n\in\eM$. Then, $n\ge 10 >7$ so by   Lemma~\ref{lem01},
\begin{align*}
   S_{4,1}(n+1)-q_{4,1}(n+2)&=S_{4,1}(n)-q_{4,1}(n+1)+2q_{4,1}(n+1)-q_{4,1}(n+2)\\
   &\geq
S_{4,1}(n)-q_{4,1}(n+1)\ge 244,
\end{align*}
and therefore $n+1\in \eM$.

(b) Let now $\eM$ denote the set of all integers $n\ge 8$ such that the statement (b)  of the lemma is valid. The table shows that
$S_{4,3}(8)-q_{4,3}(8)=184 -59= 125$, so $8\in\eM$. The rest proceeds as above.
\end{proof}

\begin{theorem}  \label{theo01}
(a) For every integer $m\geq 122$, there exist distinct prime numbers  $p_{j_1},\ldots,p_{j_k}$ $ \in \eP(4,1)$
($k\geq 1$) such that $m=p_{n_{j_1}}+\cdots+p_{j_k}$.

(b) For every integer $m\geq 56$, there exist distinct prime numbers  $q_{j_1},\ldots,q_{j_\ell} \in \eP(4,3)$
($\ell\geq 1$) such that $m=q_{j_1}+\cdots+q_{j_\ell}$.
\end{theorem}

\begin{proof}
(a) Let $(k_0,n_0)=(122,13)$. By Lemma~\ref{lem02}, $S_{4,1}(n)-q_{4,1}(n+1)\ge 2k_0=244$ for all $n\ge n_0$.
Also, a direct, though tedious, computation shows that the generating polynomial
$$\prod_{i=1}^{13}(1+x^{q_{4,1}(i)})=
(1 + x^5) (1 + x^{13}) \cdots (1 + x^{101}) (1 +  x^{109})=\sum_{i=0}^{S_{4,1}(13)}\gamma_{m}(13) x^m$$
satisfies $\gamma_m(13)\ge 1$ for all $m$ in interval $[122,603]=[122,725-122]=[k_0,S_{4,1}(n_0)-k_0]$
(we remark that $\gamma_{121}(13)=\gamma_{604}(13)=0$). The rest follows from \Cref{theo02}.

(b)  Let now $(k_0,n_0)=(56,12)$. By \Cref{lem02}, $S_{4,3}(n)-q_{4,3}(n+1)\ge 2k_0=112$ for all $n\ge n_0$.
Also, a direct computation shows that the generating polynomial
$$\prod_{i=1}^{12}(1+x^{q_{4,3}(i)})=
(1 + x^3) (1 + x^{7}) \cdots (1 + x^{71}) (1 +  x^{79})=
\sum_{i=0}^{S_{4,3}(12)}\gamma_{m}(12) x^m$$
satisfies $\gamma_m(12)\ge 1$ for all $m$ in interval $[56,404]=[56,460-56]=[k_0,S_{4,3}(n_0)-k_0]$ (here,
$\gamma_{55}(12)=\gamma_{405}(12)=0$). The rest again follows from \Cref{theo02}.
\end{proof}

\begin{corollary} \label{cor01}
For every integer $m\geq 122$, there exist $k, l\in \bN$ and distinct prime numbers $p_1,\ldots, p_k\in \eP(4,1)$
and $q_1,\ldots, q_l\in \eP(4,3)$ such that
$m=p_1+\cdots+p_k=q_1+\cdots+q_l.$\qed
\end{corollary}

Table~\ref{Tab:2} shows that in addition to the above
corollary, every integer which does not belong to the union $ [1, 17]\cup
[19, 21]\cup [23, 28]\cup [31, 33]\cup [35, 36]\cup [38, 40]\cup [43, 45]\cup
[48, 49] \cup [51, 52]\cup[55, 57]\cup  \cup \{60, 62, 65, 68, 69, 77, 80,
81, 85, 93, 121\}$ can also be written as a sum of different primes from the
class $\eP(4,1)$ as well as the sum of different primes from the class
$\eP(4,3)$.
\begin{table}[h!]
{\tiny
\begin{tabular}{|l|l|l|}
\hline
 18=5+13=7+11\hspace{24mm} & 22=5+17=3+19& 29=29=3+7+19 \\ \hline
 30=13+17=7+23 &  34=5+29=3+31 \hspace{24mm} & 37=37=3+11+23 \\ \hline
 41=41=3+7+31 & 42=5+37=11+31 & 46=5+41=3+43 \hspace{24mm} \\ \hline
 47=5+13+29=47 & 50=13+37=3+47 & 53=53=3+7+43 \\ \hline
 54=13+41=7+47 & 58=5+53=11+47 & 59=5+13+41=59 \\ \hline
 61=61=3+11+47 & 63=5+17+41=3+7+11+19+23 & 64=5+13+17+29=3+7+11+43 \\ \hline
 66=5+61=7+59  & 67=13+17+37=67 & 70=17+53=3+67 \\ \hline
 71=5+13+53=71 &  72=5+13+17+37=3+7+19+43 & 73=73=3+11+59 \\ \hline
 74=13+61=3+71 & 75=5+17+53=3+7+11+23+31 & 76=5+13+17+41=3+7+19+47 \\ \hline
 78=5+73=7+71 & 79=5+13+61=79  & 82=29+53=3+79 \\ \hline
 83=5+17+61=83 &  84=5+13+29+37=3+7+31+43 & 86=13+73=3+83 \\ \hline
 87=5+29+53=3+7+11+19+47 & 88=5+13+17+53=3+7+11+67 & 89=89=3+7+79 \\ \hline
 90=17+73=7+83 & 91=5+13+73=3+7+11+23+47 & 92=5+17+29+41=3+7+11+71 \\ \hline
 94=5+89=11+83 & 95=5+17+73=3+7+11+31+43 & 96=5+13+17+61=3+7+19+67 \\ \hline
 97=97=3+11+83  & 98=37+61=19+79 & 99=5+41+53=3+7+11+19+59 \\ \hline
 100=5+13+29+53=3+7+11+79 & 101=101=3+19+79 & 102=5+97=19+83 \\ \hline
 103=5+37+61=3+7+11+23+59 & 104=5+17+29+53=3+7+11+83 & 105=5+13+17+29+41=3+19+83 \\ \hline
 106=5+101=23+83 & 107=5+13+89=3+7+11+19+67 & 108=5+13+17+73=3+7+19+79 \\ \hline
 109=109=3+23+83 & 110=13+97=31+79 & 111=5+17+89=3+7+11+19+71 \\ \hline
 112=5+13+41+53=3+7+19+83 & 113=113=3+31+79 & 114=5+109=31+83 \\ \hline
 115=5+13+97=3+7+11+23+71 & 116=5+13+37+61=3+7+23+83 & 117=5+13+17+29+53=3+31+83 \\ \hline
 118=5+113=47+71 & 119=5+13+101=3+7+11+19+79 & 120=5+13+29+73=3+7+31+79   \\  \hline
\end{tabular} \vspace{1mm}
\caption{\label{Tab:2}}}
\end{table}

By using also other residuum classes of primes we can get additional integers
that can be expressed as a sum of different primes from two disjoint subsets
(Table~\ref{Tab:3}).
\begin{table}[h!]
{\tiny
\begin{tabular}{|l|l|l|}
\hline
16=3+13=5+11 \hspace{22mm} &  19=19=3+5+11  & 20=3+17=7+13 \\ \hline
23=23=3+7+13  & 24=5+19=7+17 \hspace{22mm}  & 26=3+23=7+19 \\ \hline
28=5+23=11+17  & 31=31=3+5+23 & 32=3+29=13+19 \\ \hline
33=3+7+23=5+11+17 & 35=3+13+19=5+7+23 & 36=5+31=7+29 \\ \hline
38=7+31=3+5+11+19  & 39=3+5+31=7+13+19 & 40=3+37=11+29 \\ \hline
43=43=3+11+29  & 44=3+41=7+37 & 45=3+11+31=5+17+23 \\ \hline
48=5+43=7+41  & 49=3+5+41=7+11+31 & 51=3+5+43=7+13+31 \\ \hline
52=5+47=11+41  & 55=3+5+47=7+11+37 & 56=13+43=19+37 \\ \hline
57=3+7+47=5+11+41  & 60=13+47=17+43 & 62=19+43=3+5+7+47 \\ \hline
65=3+19+43=5+13+47 & 68=31+37=3+5+13+47 & 69=3+19+47=5+23+41 \\ \hline
77=3+31+43=7+23+47 & 80=37+43=3+5+31+41 & 81=3+31+47=11+29+41 \\ \hline
85=5+37+43=7+31+47 & 93=3+43+47=23+29+41 & 121=31+43+47=3+11+29+37+41 \\ \hline
\end{tabular}  \vspace{1mm}
\caption{\label{Tab:3}}
}
\end{table}

Notice that Tables \ref{Tab:2} and \ref{Tab:2} together with \Cref{cor01}
show that all integers except those in $\{1, 2, 3, 4, 5, 6, 7, 8, 9, 10, 11,
12, 13, 14, 15, 17, 21, 25, 27\}$ can be written as a sum of distinct primes
in at least two ways such that no prime appears in two different sums.

\begin{corollary} \label{cor02}
Let  $n\ge 4$. If $n$ and $n-1$ are not prime numbers, then $\diam\Delta^o(\pS_n)^\ast=3$.
\end{corollary}

\begin{proof}
By \Cref{cor01} and the tables above, we see that for every nonprime number $n\ge 16$, except for  $n= 21, 25, 27$,
there are disjoint subsets $\eT_1$ and $\eT_2$ of the sequence of primes $\eP$, neither containing the prime $2$,
such that $n=\sum\limits_{p\in \eT_1}p=\sum\limits_{q\in \eT_2}q$. For those integers, the claim follows directly from
\Cref{prop}. The remaining integers $n$, for which neither $n$ nor $n-1$ is a prime number, are $n=9,10,15,21,25,27$.
Except for $n=15,21$, they are all either powers of a prime number or are immediate successors of powers of a prime
number and the claim follows from~\cite[Corollary 4.13]{Dalal-Mukherjee-Patra} ($n=15$ is also considered there).
If $n=21$, then $n-1=20=3+17=7+13$ and \Cref{prop} again is applicable.
\end{proof}

\subsection*{Acknowledgment}
The authors were partially supported by the Slovenian Research and Innovation Agency through the research programs
P2-0268, P1-0285, and research projects N1-0210, N1-0296, and J1-50000.

	
\end{document}